\documentclass[reqno,11pt]{amsart}
\usepackage{color}
\usepackage{mathtools}
\usepackage{amsfonts}
\usepackage{amssymb}
\usepackage{amsthm}
\usepackage{newlfont}
\usepackage{graphicx}
\usepackage{float}
\usepackage{amsmath}
\usepackage{geometry}
\usepackage{tikz}
\usetikzlibrary{matrix,shapes,arrows,positioning,chains}

\numberwithin{equation}{section}
\newtheorem{thm}{Theorem}[section]
\newtheorem{lem}{Lemma}[section]
\newtheorem{rem}{Remark}[section]

\newtheorem{example}{Example}[section]

\newcommand{\ed}{\end {document}}

\begin{document}

\title{On a sinc-type MBE model}

\author[X.Y. Cheng]{Xinyu Cheng}
\address{X.Y. Cheng, Department of Mathematics, University of British Columbia, Vancouver, BC V6T 1Z2, Canada}
\email{xycheng@math.ubc.ca}

\author[D. Li]{ Dong Li}
\address{D. Li, Department of Mathematics, the Hong Kong University of Science \& Technology, Clear Water Bay, Kowloon, Hong Kong}
\email{mpdongli@gmail.com}

\author[C.Y. Quan]{Chaoyu Quan}	
\address{C.Y. Quan, SUSTech International Center for Mathematics, Southern University of Science and Technology,
	Shenzhen, P.R. China}
\email{quancy@sustech.edu.cn}

\author[W. Yang]{Wen Yang}
\address{\noindent W. Yang,~Wuhan Institute of Physics and Mathematics, Chinese Academy of Sciences, P.O. Box 71010, Wuhan 430071, P. R. China; Innovation Academy for Precision Measurement Science and Technology, Chinese Academy of Sciences, Wuhan 430071, P. R. China.}
\email{wyang@wipm.ac.cn}

\begin{abstract}
We introduce a new sinc-type molecular beam epitaxy model which is derived from a cosine-type
energy functional. The landscape of the new functional is remarkably similar to the classical MBE model
with double well potential but has the additional advantage that all its derivatives are uniformly
bounded. We consider first order IMEX and second order BDF2 discretization schemes. For both cases
we quantify  explicit time step constraints for the energy dissipation which is in good accord with the practical
numerical simulations. Furthermore we introduce a new theoretical framework and prove unconditional
uniform energy boundedness with no size restrictions on the time step. This is the first unconditional (i.e. independent of the time step size)
result for semi-implicit methods applied to the  phase field models without introducing any artificial stabilization terms or fictitious variables.
\end{abstract}

\maketitle

\section{Introduction}
In this work we consider the following molecular beam epitaxy (MBE) model posed
on the two dimensional periodic torus $\Omega = \mathbb T^2=[-\pi, \pi]^2$:
\begin{align} \label{1.1}
\partial_t h
= -\eta^2 \Delta^2 h   -\nabla \cdot (
\mathrm{sinc}(  |\nabla h|) \nabla h), \qquad (t,x) \in (0,\infty) \times \Omega,
\end{align}
where $\mathrm{sinc}(t) = \sin t /t $ is the usual $\mathrm{sinc}$-function and
\begin{align}
|\nabla h | = \sqrt{ (\partial_{x_1} h)^2 + (\partial_{x_2} h)^2}.
\end{align}
 The function $h=h(t,x)$ represents a scaled height function of thin film  and $\eta^2$ is a positive
constant.  Note that since the power series of $\mathrm{sinc}(t)$ consists only even powers of $t$, the term
$\mathrm{sinc}(|\nabla h|)$ is a bounded smooth function in $|\nabla h|^2=\nabla h\cdot \nabla h$.
Furthermore
\begin{align}
\sup_{z\in \mathbb R^2} |\partial_{z_1}^{k_1} \partial_{z_2}^{k_2} (\mathrm{sinc}(|z|) z) |
<\infty, \qquad\forall\, k_1, k_2\ge 0.
\end{align}
 Thanks to this simple observation it is a breeze
to build a wellposedness theory for \eqref{1.1}.
The equation \eqref{1.1} can be regarded as a $L^2$ gradient flow  of the energy functional
\begin{align} \label{1.3}
\mathcal E (h)= \int_{\Omega}
\Bigl( \frac 12 \eta^2 | \Delta h |^2 +  \cos ( |\nabla h | )  \Bigr) dx.
\end{align}
Consequently for smooth solutions, we have the energy identity
\begin{align}
\mathcal E( h(t_2) ) + \int_{t_1}^{t_2} \| \partial_t h \|_2^2 dt = \mathcal
E (h(t_1) ), \qquad\forall\, 0\le t_1 <t_2 <\infty.
\end{align}
In particular we have
\begin{align}
\mathcal E (h(t) ) \le \mathcal E (h(0) ), \qquad\forall\, t>0.
\end{align}
By using this fundamental monotonicity law in conjunction with the smoothness
and boundedness of the nonlinearity, one can obtain global wellposedness and regularity
for \eqref{1.1} with $H^2$ (or rougher) initial data. Note that for $|z|\ll 1$, we
have
\begin{align}
\cos z =\frac 1{24} (z^2-6)^2 -\frac 12 + O(|z|^6).
\end{align}
Thus up to a harmless additive constant, the energy functional \eqref{1.3} is
 similar to the standard double well-potential (see Section 2 \eqref{2.8}--\eqref{2.9}), namely
\begin{align}  \label{E_st}
\mathcal E_{\mathrm{st} }(h) &= \int_{\Omega}
\Bigl( \frac 12 \eta^2 |\Delta h|^2 + \frac 1{24}(|\nabla h|^2-6)^2 \Bigr) dx.
\end{align}
In the MBE literature, it is well-known that the energy functional
\eqref{E_st}  leads to the thin film model with slope selection (cf.
\cite{ll2003} and \cite{xt2006}).
However a fundamental difficulty with the analysis and simulation of the dynamical
equation corresponding to
\eqref{E_st} is the lack of good  Lipschitz bounds on the nonlinearity (see \cite{lwy2020,lqt2016}).
In stark contrast our model \eqref{1.1} and its cousins has a ``built-in" control of the Lipschitz
norms which renders the analysis and simulation much more appealing. Moreover
 these new models satisfy Onsager's reciprocity relations and capture much the same physics
 as prototypical MBE models. To discuss these connections and put things into perspective we review below some prior arts
 and related analysis.

Molecular beam epitaxy (MBE) is a process for growing thin, epitaxial films of a wide variety of materials, ranging from oxides to semiconductors and metals. It is first used in  the epitaxial growth of compound semiconductor films by a process involving the reaction of one or more thermal molecular beams with a crystalline surface under ultra-high vaccum conditions. These arriving constituent atoms form a crystalline layer in registry with the substrate, i.e., an epitaxial film. To understand the  mechanism of this physical process, several different PDE models have been proposed  for investigating the morphological instability and interfacial dynamics at various temporal and spatial scales \cite{k1997,v1991}. In a fixed Cartesian coordinate system, one can denote $h_{\mathrm{real}}$
as the actual height of the film at time $t$, and traces the evolution of $h(t,x)= h_{\mathrm{real}}
-F t$ where $F$ stands for the deposition flux.  In yet other words the film surface at time $t$
is represented by $x_3= h(t,x_1,x_2)$ in a co-moving frame. By using conservation of mass, we have
\begin{equation} \label{cons_1}
h_t=-\Omega_a\nabla\cdot J,
\end{equation}
where $\Omega_a$ is the atomic volume and $J$ is the surface current. On a purely phenomenological basis and considering only isotropic growth, we may represent the surface current as a series (cf.
  \cite{ors1999,z1995,z1996}
  )
\begin{equation}
\label{1.sc}
J=A_1\nabla h+A_2\nabla(\Delta h)+A_3|\nabla h|^2\nabla h+A_4\nabla|\nabla h|^2+\cdots,
\end{equation}
In \eqref{1.sc} one typically only retains terms vanishing no faster than
$L^{-4}$. Plugging this into \eqref{cons_1}, we arrive at
\begin{align}
\partial_t h=~&-\Omega_a \Bigl( A_1\Delta h+A_2\Delta^2h+A_3\nabla\cdot(|\nabla h|^2\nabla h)+A_4\Delta (|\nabla h|^2) \Bigr).
\end{align}
The physical meaning of each term is as follows:
\begin{itemize}
\item $A_1 \Delta h$: diffusion by particle exchange between surface and
the vapor (Mullins \cite{m1957});
\item  $A_2 \Delta^2 h$: capillarity-driven surface diffusion (Mullins \cite{m1957},
Herring \cite{h1951});
\item $A_3 \nabla \cdot (|\nabla h|^2 \nabla h)$: consistent with an atomistic
model by Das Sarma-Ghaisas \cite{dg1992};
\item $A_4 \Delta (|\nabla h|^2)$: Lai-Das Sarma \cite{ld1991} used it to describe the model where an atom moves to a neighboring kink site and breaks the bond to find another kink with smaller step height.
\end{itemize}
As it was well explained in \cite{ld1991}, each of these terms can be viewed as a building block
to form composite structures such as peaks, valeys and so on (see figure \ref{fig:ders} for the
1D case).
In \cite{ors1999}, the term $A_4 \Delta (|\nabla h|^2)$ was dropped in view of
Onsager's reciprocity relations (\cite{onsager1931,onsager1931-2,pb1996}).  By a relabelling
of the constants,  this leads
to the standard model
\begin{align}
\partial_t h = - D \Delta^2 h +v \nabla \cdot ( (k^{-2} |\nabla h|^2 -1) \nabla h).
\end{align}

\begin{figure}[!h]
\centering
\includegraphics[width=0.99\textwidth,clip,trim={2in 0.8in 1.5in 0.5in}]{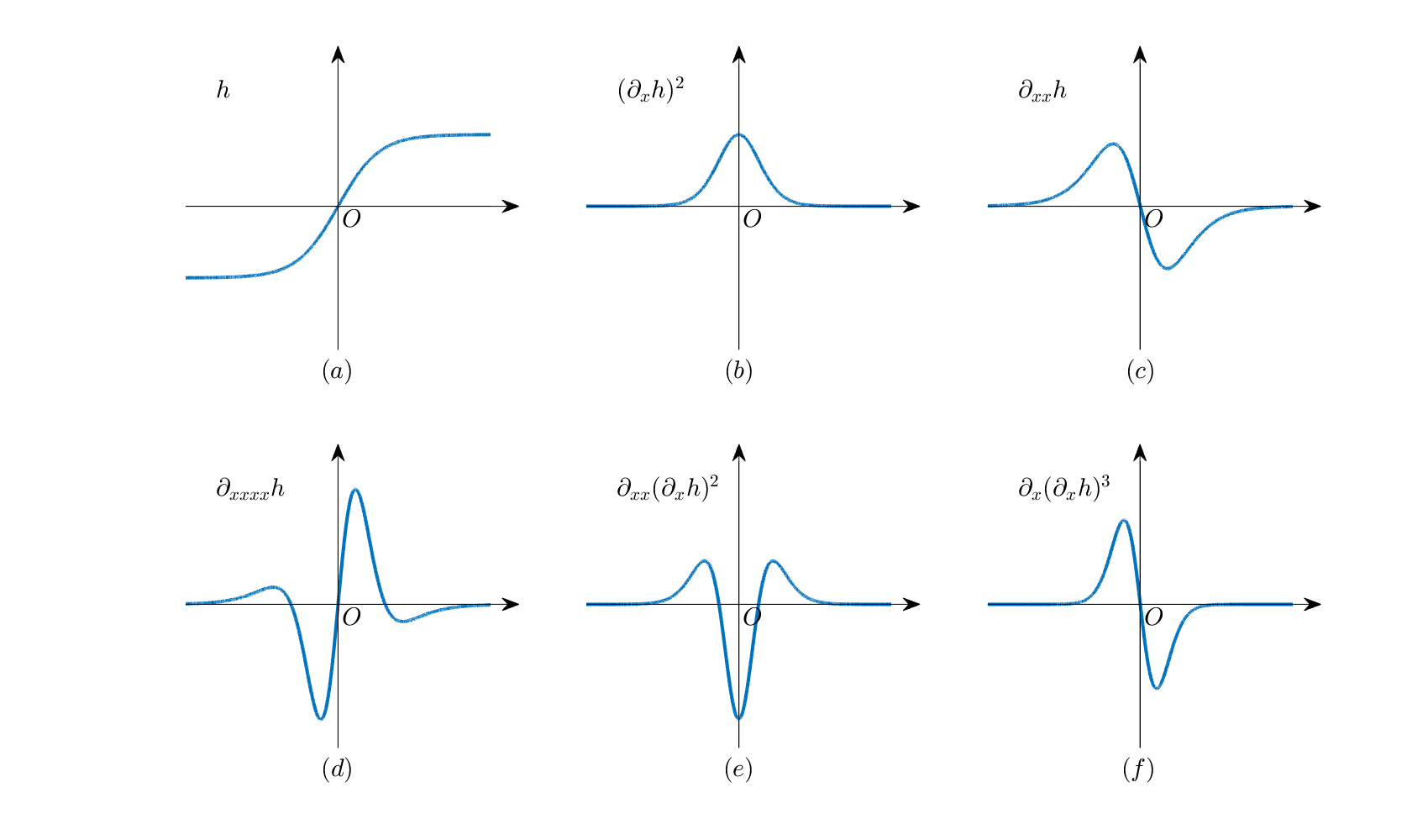}
\caption{\small Schematic diagram of $h = \tanh(x)$ and its derivatives.  }\label{fig:ders}
\end{figure}

In Section 2, we shall show that our model \eqref{1.1} corresponds to the expansion
\eqref{1.sc} by carefully incorporating higher order terms which satisfy Onsager's reciprocity
relations. In the regime $|\nabla h|\ll 1$ our model coincides with the standard MBE model
with slope selection. On the other hand, it is \emph{vastly different} from the MBE model
without slope selection (cf. \cite{ll2003})
\begin{align} \label{1.13}
\partial_t h = - \eta^2 \Delta^2 h - \nabla \cdot ( \frac {\nabla h} {1+ |\nabla h|^2} ),
\end{align}
since its energy functional $\mathcal E (h) = \int_{\Omega}
(\frac {\eta^2}2 |\Delta h|^2  - \frac 12 \ln (|1+ |\nabla h|^2) )dx$ behaves badly and leads
to mound-like structures.  Note that a nice feature of \eqref{1.13} is that its nonlinearity has
bounded derivatives of all orders. However its energy landscape is different from the standard MBE
model with slope selection. Our model can be regarded as a variant of the standard MBE model with slope
selection and with \emph{tamed} nonlinearity (in terms of Lipschitz control).  In future works we plan to carry out an in-depth comparative study of our new model with existing standard models.

The main contribution of this work is as follows.
\begin{enumerate}
\item  We introduce a new sinc-type molecular beam epitaxy model which is derived from a cosine-type energy functional.   The landscape of the new functional is remarkably similar to the classical MBE model with double well potential but has the additional advantage that all its derivatives are uniformly
bounded. Furthermore it captures the same physics as the classical MBE model with slope selection.
It is possible to generalize our approach to a more systematic construction of ``benign" free energies whilst keeping the same physics.

\item We consider first order IMEX and second order BDF2 discretization schemes. For both cases
we quantify  explicit time step constraints for the energy dissipation which is in good accord with the practical
numerical simulations.

\item We introduce a new theoretical framework and prove unconditional
uniform energy boundedness \emph{with no size restrictions on the time step}. This is the first unconditional (i.e. independent of the time step size)
result for semi-implicit methods applied to the  phase field models without introducing any artificial stabilization terms or fictitious variables.

\end{enumerate}

The rest of this paper is organized as follows. In Section 2 we give a more detailed derivation
of the model and explain its connection with earlier models. In Section 3 and 4 we consider
first order IMEX discretization and prove the energy stability bounds. In Section 5 and 6 we carry
out the analysis for BDF2 methods. Section 7 collects several numerical simulations with side
by side comparisons with existing MBE models. In the final section we give concluding remarks.


\section{Derivation of the model}
To allow some generality, we start with the energy functional
\begin{align}
\mathcal E = \int_{\Omega} \Bigl( \frac 12 \eta^2 |\Delta h|^2 +  \beta_1 \cos(\beta |\nabla h| )\Bigr) dx,
\end{align}
where $\beta_1$, $\beta$ are constants.

Denote
\begin{align}
g_1(t) = \sum_{k=0}^{\infty} \frac {(-1)^k t^k} {(2k)!} = 1 + \sum_{k=1}^{\infty} \frac {(-1)^k t^k }{(2k)!}.
\end{align}
Note that
\begin{align}
&g_1^{\prime}(t) = \frac 12 \sum_{k=1}^{\infty} \frac {(-1)^{k}} {(2k-1)!} t^{k-1}, \notag \\
& g_1^{\prime}(s^2) = -\frac 12 s^{-1} \sin s. \label{2.4Tp}
\end{align}
Here the identity \eqref{2.4Tp} also holds at $s=0$ provided we identify $\frac {\sin s }s \Bigr|_{s=0}=1$.

Clearly
\begin{align}
& \mathcal F = \int_{\Omega} \cos(\beta |\nabla h|)  dx = \int_{\Omega} g_1(\beta^2 |\nabla h|^2) dx, \\
&\frac {\delta \mathcal F} {\delta h} \Bigr|_{L^2}= -2\beta^2\nabla \cdot (g_1^{\prime}( \beta^2|\nabla h|^2) \nabla h),
\end{align}
where $\frac {\delta \mathcal F } {\delta h} \Bigr|_{L^2}$ denotes the variational derivative
of the functional $\mathcal F$ in $L^2$.  By \eqref{2.4Tp}, it follows that
\begin{align}
\frac {\delta \mathcal F} {\delta h} \Bigr|_{L^2}= \beta^2\nabla \cdot ( \frac {\sin(\beta |\nabla h|)}
{\beta|\nabla h|} \nabla h).
\end{align}

Thus
\begin{align}
\frac {\delta \mathcal E}{\delta h}
\Bigr|_{L^2}= \eta^2 \Delta^2 h  +\beta^2 \beta_1\nabla \cdot ( \mathrm{sinc} ( \beta |\nabla h|)
 \nabla h).
\end{align}
Note that for $|\nabla h| \ll 1$, we have
\begin{align} \label{2.8}
\mathrm{sinc} ( \beta |\nabla h|) = 1- \frac 16\beta^2 |\nabla h|^2 + O(|\nabla h|^4).
\end{align}
Taking $\beta_1=1/6$ and $\beta=\sqrt 6$, we then recover the usual MBE model with slope selection:
\begin{align} \label{2.9}
\partial_t h = -\frac {\delta \mathcal E}{\delta h}
\Bigr|_{L^2}= -\eta^2 \Delta^2 h  -\nabla \cdot ( (1- |\nabla h|^2) \nabla h).
\end{align}

We now explain the connection with the general model mentioned in the introduction especially
concerning the expansion \eqref{1.sc}.
\subsection{Connection with the general model.}
Firstly, we start with the expansion
\begin{equation}
\label{1.sc-f}
\begin{aligned}
J=~&A_1\nabla h+A_2\nabla(\Delta h)+A_3|\nabla h|^2\nabla h+A_4\nabla|\nabla h|^2+\sum_{i=1}^\infty A_{4+i}|\nabla h|^{2+2i}\nabla h.
\end{aligned}
\end{equation}
Substituting \eqref{1.sc-f} into \eqref{1.sc} we get
\begin{equation}
\label{1.h-f}
\begin{aligned}
h_t=~&-\Omega_a\Big[A_1\Delta h+A_2\Delta^2h+A_3\nabla\cdot(|\nabla h|^2\nabla h)+A_4\Delta|\nabla h|^2\\ &+\sum_{i=1}^\infty A_{4+i}\nabla\cdot(|\nabla h|^{2+2i}\nabla h|)\Big].
\end{aligned}
\end{equation}
Enforcing Onsager's reciprocity relations (see \cite{onsager1931,onsager1931-2,pb1996}), we can drop the  $A_4$ term. Thus
\begin{equation}
\label{1.h-e}
\partial_th=-c\nabla\cdot(\nabla h(1-\sum_{i=1}^\infty B_i|\nabla h|^{2i}))-D \Delta^2h,
\end{equation}
where $c=\Omega_aA_1$, $B_1=-\frac{A_3}{A_1}$, $B_i=-\frac{A_{3+i}}{A_1}$ for $i\geq2$ and $D=\Omega_aA_2$. When the coefficients $B_i$ satisfy
\begin{equation}
B_i=\frac{(-1)^{i-1}}{(2i+1)!}\quad\mbox{for}\quad i\geq1,
\end{equation}
we obtain
\begin{equation}
\label{1.smbe}
\partial_t h=-D\Delta^2h-c\nabla\cdot\left(\frac{\sin(|\nabla h|)}{|\nabla h|}\nabla h\right).
\end{equation}

\section{Energy decay of first order IMEX}
We consider the following first order implicit-explicit (IMEX) scheme:
\begin{align} \label{s3.1}
\frac {h^{n+1}-h^n}{\tau}
= -\eta^2 \Delta^2 h^{n+1} + \nabla \cdot ( g(\nabla h^n) ),
\end{align}
where
\begin{align}
g(z) =- \frac {\sin (|z|) }{|z|} z, \quad z\in \mathbb R^2.
\end{align}

\begin{lem} \label{Ls3.2}
Let $G(z) = \cos |z| =\cos \sqrt{z_1^2+z_2^2}$ for $z \in \mathbb R^2$. Then
\begin{align}
 \sum_{i,j=1}^2  x_i x_j \partial_{z_i} \partial_{z_j} G (z)
\le  |x|^2, \qquad \forall\, x, z\in\mathbb R^2.
\end{align}
\end{lem}
\begin{proof}
With no loss we assume $z\ne 0$.
By a simple computation we have
\begin{align} \label{e3.5_0}
\partial_{z_i} \partial_{z_j} G
=-\left(\cos |z|- \frac{\sin |z|} {|z|}\right)
\frac{z_i z_j}{|z|^2} -\frac {\sin |z|} {|z|} \delta_{ij}.
\end{align}
Thus
\begin{align}
 \sum_{i,j=1}^2  x_i x_j \partial_{z_i} \partial_{z_j} G (z)
&= \left(-\cos |z|+ \frac{\sin |z|} {|z|}\right)
\frac{|x\cdot z|^2}{|z|^2} -\frac {\sin |z|} {|z|} |x|^2 \notag \\
&= (-\cos |z|) \frac{ |x\cdot z|^2} {|z|^2}
-\frac{\sin |z|} {|z|} |x\cdot z^{\perp}|^2,
\end{align}
where $z^{\perp}$ is a unit vector orthogonal to $z$.
The desired result then easily follows.
\end{proof}

For simplicity of notation,  we denote
\begin{align}
E_n = \int_{\Omega}\Bigl(
\frac {\eta^2}2 |\Delta h^n|^2 + \cos |\nabla h^n | \Bigr)dx.
\end{align}

\begin{thm}[Energy dissipation]\label{thm3.1}
Consider the IMEX scheme \eqref{s3.1} with $h^0 \in H^2(\mathbb T^2)$.
For any $n\ge 0$,
\begin{align}
 &
 E_{n+1}-E_n  +( \sqrt{\frac{2\eta^2}{\tau}} -\frac {1}2) \|\nabla(h^{n+1}-h^n)\|_2^2
 \le 0,
 \end{align}
 In particular if
 \begin{align}
0< \tau \le {8\eta^2},
 \end{align}
 then $E_{n+1}\le E_n$ for all $n\ge 0$.
\end{thm}
\begin{proof}
Taking the $L^2$-inner product with $(h^{n+1}-h^n)$ on both sides of \eqref{s3.1}, we get
\begin{align*}
 \frac 1 {\tau} \| h^{n+1}-h^n \|_2^2 + \frac {\eta^2}2 ( \| \Delta h^{n+1}\|_2^2 -\| \Delta h^n \|_2^2
   + \| \Delta(h^{n+1}-h^n) \|_2^2 )  
   = - (g(\nabla h^n), \nabla (h^{n+1}-h^n) ).
\end{align*}

Denote $G(z)=\cos |z| = \cos \sqrt{z_1^2+z_2^2}$ for $z\in \mathbb R^2$. Introduce
\begin{align*}
H(s)= G(\nabla h^n + s(\nabla h^{n+1} -\nabla h^n) ).
\end{align*}

By using the expansion
\begin{align*}
H(1) = H(0) + H^{\prime}(0) + \int_0^1 H^{\prime\prime}(s) (1-s) ds,
\end{align*}
we get
\begin{align*}
G(\nabla h^{n+1})&-G(\nabla h^n) = g(\nabla h^n)\cdot (\nabla h^{n+1}-\nabla h^n) \notag \\
&\; + \sum_{i,j=1}^2 \partial_i(h^{n+1}-h^n) \partial_j (h^{n+1}-h^n)
\int_0^1 (\partial_{ij} G)(\nabla h^n +s(\nabla h^{n+1}-\nabla h^n) ) (1-s) ds.
\end{align*}
By Lemma \ref{Ls3.2}, we obtain
\begin{align}
E_{n+1}-E_n+ \frac 1 {\tau} \| h^{n+1}-h^n \|_2^2 +\frac {\eta^2}2\| \Delta(h^{n+1}-h^n) \|_2^2
   \le \frac 12 \| \nabla (h^{n+1}-h^n) \|_2^2.
\end{align}

The desired inequality now follows from this and the simple interpolation inequality
\begin{align}
\|\nabla h\|_2 \le \| h\|_2^{\frac 12} \| \Delta h\|_2^{\frac 12}.
\end{align}
\end{proof}

\section{Uniform boundedness of the energy for any $\tau>0$}

\begin{thm}[Unconditional uniform energy boundedness] \label{thm4.1}
Consider the scheme \eqref{s3.1} with $\tau>0$ and $h^0 \in H^2(\mathbb T^2)$.   We have
\begin{align}
\sup_{n\ge 1}(\| h^n\|_2+ \| \Delta h^n\|_2) \le C_1<\infty,
\end{align}
where $C_1>0$ depends only on $(h^0, \eta)$. Note that $C_1$ is \underline{independent of $\tau$}.
\end{thm}
\begin{proof}
We begin by noting that
\begin{align}
\int_{\Omega} h^n dx =\int_{\Omega} h^0 dx, \qquad \forall\, n\ge 1.
\end{align}
Thus with no loss we may assume $\int_{\Omega} h^n dx =0$ for all $n\ge 0$. This will justify the use
of the Poincar\'e inequality in our proof below.

Taking the $L^2$-inner product with $h^{n+1}$ on both sides of \eqref{s3.1}, we obtain
\begin{align}
\frac {\|h^{n+1}\|_2^2 -\| h^n\|_2^2 +\| h^{n+1}-h^n\|_2^2}
{2\tau} + \eta^2 \| \Delta h^{n+1} \|_2^2 = -( g(\nabla h^n), \nabla h^{n+1} ),
\end{align}
where $(\cdot,\cdot)$ denotes the usual $L^2$-inner product.  Since $\|g\|_{\infty} \le 1$, we obtain
\begin{align}
 -( g(\nabla h^n), \nabla h^{n+1} ) \le  \|\nabla h^{n+1} \|_1 \le 2\pi \| \nabla h^{n+1} \|_2
 \le 2\pi \| \Delta h^{n+1} \|_2.
 \end{align}
Therefore we obtain for all $n\ge 0$,
\begin{align}
\frac {\|h^{n+1}\|_2^2 -\| h^n\|_2^2 }{2\tau} +\eta^2 \| \Delta h^{n+1} \|_2^2
\le 2\pi \| \Delta h^{n+1} \|_2.
\end{align}
For convenience of notation, we rewrite it as
\begin{align} \label{4.5tP0}
\frac {\|h^{n}\|_2^2 -\| h^{n-1} \|_2^2 }{2\tau} +\eta^2 \| \Delta h^{n} \|_2^2
\le 2\pi \| \Delta h^{n} \|_2, \qquad\forall\, n\ge 1.
\end{align}

By using \eqref{4.5tP0}, we can derive
\begin{align} \label{4.5tP1}
\sup_{n\ge 0} \| h^{n} \|_2 \le C_2,
\end{align}
where $C_2>0$ depends only on $(h^0, \eta)$.

To show \eqref{4.5tP1}, consider the set
\begin{align}
S= \{ n \ge 1: \,  \|\Delta h^n \|_2 \le  \frac {2\pi} {\eta^2} \}.
\end{align}

Case 1: $S$ is the empty set. In this case by using \eqref{4.5tP0},  we have
\begin{align}
\| h^n \|_2 \le \| h^{n-1} \|_2, \qquad\forall\, n \ge 1.
\end{align}
This immediately yields \eqref{4.5tP1}.

Case 2: $S$ is the set of natural numbers.  Then since all $n\ge 1$ are in the set $S$, the estimate
\eqref{4.5tP1} follows from the Poincar\'e inequality.

Case 3: $S$ is a proper nonempty subset of the set of natural numbers.  We only need to show the uniform
$L^2$ estimate for any $n\notin S$.  Now consider any $n\notin S$.  We discuss two subcases.

Subcase 3a: $S\cap \{ j:\, 1\le j<n \}$ is an empty set. In yet other words, all $j<n$ are not in $S$.
Then in this case by using \eqref{4.5tP0}, we have
\begin{align}
\| h^{j} \|_2 \le \| h^{j-1} \|_2, \qquad\forall\, 1\le j \le n.
\end{align}
Thus this case is OK.

Subcase 3b: There exists some $j_*<n$ with $j_*\in S$. With no loss we assume $[j_*+1, n]\cap S$
is an empty set.  By using the Poincar\'e inequality, we have
\begin{align}
\| h^{j_*} \|_2 \le \frac {2\pi} {\eta^2}.
\end{align}
For all $j_*+1\le j \le n$, by using \eqref{4.5tP0} we have
\begin{align}
\| h^j\|_2 \le \| h^{j-1} \|_2.
\end{align}
Thus in this case we obtain
\begin{align}
\| h^n\|_2 \le \frac {2\pi} {\eta^2}.
\end{align}

Collecting the estimates, we have proved \eqref{4.5tP1}. Now to obtain the uniform $H^2$-estimate, we argue as follows.

Clearly, with no loss we can assume $\tau> 8 \eta^2$ since the case $0<\tau \le 8 \eta^2$
is already covered by Theorem \ref{thm3.1}.

We rewrite \eqref{s3.1} as
\begin{align}
\eta^2 \Delta^2 h^{n+1}
= - \frac {h^{n+1}-h^n} {\tau} + \nabla \cdot (g(\nabla h^n) ).
\end{align}
Since $\|g\|_{\infty} \le 1$, we have
\begin{align} \label{4.15tP1}
\|(-\Delta)^{-1} \nabla \cdot ( g(\nabla h^n) ) \|_2 \le O(1),
\end{align}
where $O(1)$ denotes a constant depending only on $(h^0, \eta)$.
Clearly by using \eqref{4.5tP1}, \eqref{4.15tP1} and the fact that $\tau \gtrsim 1$, we obtain
\begin{align}
\eta^2 \| \Delta h^{n+1} \|_2  \le O(1), \qquad\forall\, n\ge 0.
\end{align}
\end{proof}

\begin{rem}
Our results can be generalized to more general nonlinearities with bounded derivatives. For example,
consider the following Cahn-Hilliard type equation
\begin{align}
\partial_t u = -\Delta ( \eta^2 \Delta u - f( u) ), \qquad(t,x) \in (0,\infty)\times \mathbb T^2,
\end{align}
where $u$ is a scalar-valued function. Assume $f(z)=F^{\prime}(z)$ and
\begin{align}
\| F\|_{\infty} \le N_0<\infty, \; \|f\|_{\infty} \le N_1<\infty, \; \sup_{z\in \mathbb R}  f^{\prime} (z) \le N_2<\infty,
\end{align}
where $N_i>0$, $i=0,1,2$ are constants.
Denote
\begin{align}
\mathcal E ( u) = \int_{\mathbb T^2}
\left( \frac 12 \eta^2 \| \nabla  u\|_2^2  + F(u) \right) dx.
\end{align}
Consider the first order IMEX discretization of the form:
\begin{align}
\frac {u^{n} -u^{n-1} }{\tau}
= -\Delta ( \eta^2 \Delta u^n  - f(u^{n-1} ) ), \quad n\ge 1.
\end{align}
Assume $u^0 \in H^1(\mathbb T^2)$ with mean zero. Then
\begin{align}
\frac 1 {\tau} \| |\nabla|^{-1} ( u^{n}-u^{n-1} ) \|_2^2
+ \frac {\eta^2}2 \| \nabla (u^{n}-u^{n-1}) \|_2^2 +
\mathcal E(u^n ) - \mathcal E (u^{n-1} ) \le\frac 12 N_2 \| u^n -u^{n-1} \|_2^2, \quad\forall\, n\ge 1.
\end{align}
It follows that if $0<\tau\le \frac {8\eta^2} {N_2^2}$, then we obtain
\begin{align} \label{t4.22.0}
\mathcal E (u^n) \le \mathcal E(u^{n-1}), \qquad \forall\, n\ge 1.
\end{align}
By using \eqref{t4.22.0} (to control the regime $\tau \lesssim 1$), we can
establish the following bound which is an analogue of Theorem \ref{thm4.1}:
for any $\tau>0$, it holds that
\begin{align}
\sup_{n\ge 1 } \| u^n \|_{H^1(\mathbb T^2)} \le D_1<\infty,
\end{align}
where $D_1$ depends on  $(\eta^2, N_0,N_1, N_2, u^0)$.  Note that $D_1$ is
{independent of $\tau$}.
 \end{rem}

 \begin{rem}
 In \cite{SY10} Shen and Yang introduced effective
Lipschitz truncation of the nonlinearity together with adding suitable stabilization terms order to prove the unconditional energy stability of numerical
schemes for Allen-Cahn and Cahn-Hilliard equations.
As shown in the preceding remark, our analysis encompasses the truncated models in \cite{SY10}
as special cases. In particular, for these truncated models we obtain uniform energy bounds
\textbf{with no conditions on the time step}. Our analysis also extends to higher order in time
methods, see later sections for details.
\end{rem}

\section{Energy decay for second-order BDF2 scheme}
We consider the following BDF2 scheme:
\begin{equation}\label{eq:sch2}
\frac{3h^{n+1}-4h^n+h^{n-1}}{2\tau}=-\eta^2\Delta^2 h^{n+1}+2 \nabla \cdot (g(\nabla h^n))-
\nabla \cdot (g ( \nabla h^{n-1})), \qquad n\ge 1,
\end{equation}
where
\begin{align}
g(z) =- \frac {\sin (|z|) }{|z|} z, \quad z\in \mathbb R^2.
\end{align}

To kick start the scheme we can compute $h^1$ using a first order scheme such as \eqref{s3.1}.

\begin{lem} \label{gLip0.1}
For  $g(z) =- \frac {\sin (|z|) }{|z|} z$,  we have
\begin{align}
|g(x)-g(y) | \le  |x-y|, \qquad\forall\, x, y \in \mathbb R^2,
\end{align}
Here $|z|=\sqrt{z_1^2+z_2^2}$ is the usual $l_2$-norm on $\mathbb R^2$.
\end{lem}
\begin{proof}
By using the Fundamental Theorem of Calculus, we have
\begin{align}
g(x)-g(y) = \int_0^1 (D g) (y+\theta(x-y) ) d\theta (x-y).
\end{align}
By \eqref{e3.5_0}, we have
\begin{align}
(Dg )(z)=\left(-\cos |z|+ \frac{\sin |z|} {|z|}\right)
\frac{z_i z_j}{|z|^2} -\frac {\sin |z|} {|z|} \delta_{ij}.
\end{align}
Consider the matrix $(Dg)(z)$ and denote $s=|z|$. It is not difficult to check that its
eigen-values are given by
\begin{align}
\lambda_1=-\cos s, \quad \lambda_2 = -\frac{\sin s} s.
\end{align}

 The desired result then follows.
\end{proof}

For the second order scheme \eqref{eq:sch2},
we shall establish the dissipation law for a slightly modified energy. Namely denote
\begin{align} \label{EnMod}
\widetilde E_n  &= E_n+\frac 1{4\tau} \|  h^{n}- h^{n-1}\|_2^2 +\frac {1} 2
\| \nabla h^n-\nabla h^{n-1} \|_2^2\notag \\
&=\frac 12 \eta^2 \|\nabla h^{n}\|_2^2 +\int_{\mathbb T^2} \cos|\nabla h^n| dx+\frac 1{4\tau} \| h^{n}- h^{n-1}\|_2^2+\frac {1}2
\| \nabla h^n-\nabla h^{n-1} \|_2^2.
\end{align}
Note that if we have uniform control of all $\nabla h^n$, it is reasonable to expect that
$\|\nabla h^n - \nabla h^{n-1} \|_2 = O(\tau)$ which gives
\begin{align}
|\widetilde E_n -E_n| = O(\tau).
\end{align}
This suggests that the modified energy $\widetilde E_n$ is a fairly good approximation of the original
energy $E_n$.

\begin{thm}[Energy dissipation]\label{thm5.1}
Consider the scheme \eqref{eq:sch2}. Assume  $h^0$, $h^1 \in H^2(\mathbb T^2)$.

Assume
\begin{align}
0<\tau \le \alpha_1 \eta^2, \qquad\alpha_1=\frac 89 \approx 0.8889.
\end{align}
Then
\begin{equation}
\widetilde E_{n+1} \le  \widetilde E_{n}, \quad \forall\, n\ge 1,
\end{equation}
where 
\begin{align}
\widetilde E_n  &= E_n+\frac 1{4\tau} \|  h^{n}- h^{n-1}\|_2^2 +\frac {1}{2}
\| \nabla h^n-\nabla h^{n-1} \|_2^2.
\end{align}
Furthermore, if 
\begin{align} \label{furt001}
\frac 1 {\tau} \| h^1- h^0\|_2^2 \le \alpha_2,
\end{align}
where $\alpha_2>0$ is a constant, then 
\begin{align} \label{5.13}
\sup_{n\ge 2} ( \| h^n \|_2 + \| \Delta h^n \|_2 ) \le \widetilde{C}_1 <\infty,
\end{align}
where $\widetilde{C}_1$ depends only on ($\eta$, $h^0$, $h^1$, $\alpha_2$).
\end{thm}
\begin{rem}
We note that the assumption \eqref{furt001} is quite reasonable since typically
$h^1- h^0 =O(\tau)$ if $h^1$ is computed by the first order IMEX scheme.
\end{rem}

\begin{proof}

Denote
\begin{align}
\delta h^n = h^n -h^{n-1}.
\end{align}

Taking the $L^2$-inner product with $\delta h^{n+1}$ on both sides of \eqref{eq:sch2}, we obtain
\begin{align}\label{eq:bdf2int}
 & (\frac{3h^{n+1}-4h^n+h^{n-1}}{2\tau}, \delta h^{n+1})
+\frac 1{2} \eta^2( \| \Delta h^{n+1} \|_2^2 -\| \Delta h^n \|_2^2
 + \| \Delta( \delta h^{n+1})  \|_2^2)  \notag \\
&=\;-(g(\nabla h^n), \nabla (\delta h^{n+1}) )- (g(\nabla h^n)-g(\nabla h^{n-1}), \nabla (\delta h^{n+1}) ).
\end{align}

Observe that
\begin{align}  \label{5.9t0}
\frac{3h^{n+1}-4h^n+h^{n-1}}{2\tau}
&= \frac {h^{n+1}-h^n}{\tau} + \frac{h^{n+1}-2h^n+h^{n-1}}{2\tau} \notag \\
& = \frac {\delta h^{n+1} } {\tau} + \frac {\delta h^{n+1} -\delta h^n} {2\tau}.
\end{align}

By using \eqref{5.9t0}, we  rewrite
\begin{equation}
(\frac{3h^{n+1}-4h^n+h^{n-1}}{2\tau}, \delta h^{n+1})  = \frac 1 \tau \|\delta h^{n+1}\|^2+ \frac 1{4\tau}\left( \|\delta h^{n+1}\|^2 -  \|\delta h^{n}\|^2+  \|\delta h^{n+1}-\delta h^n\|^2\right).
\end{equation}

By using the proof of Theorem \ref{thm3.1}, we have
\begin{align}
-(g(\nabla h^n), \nabla (h^{n+1}-h^n) )\le F_n - F_{n+1} + \frac {1}2 \| \nabla (\delta h^{n+1})\|_2^2,
\end{align}
where
\begin{align}
F_n = \int_{\mathbb T^2} \cos|\nabla h^n| dx.
\end{align}

By Lemma \ref{gLip0.1}, we have
\begin{equation}
\begin{aligned}
-( g(\nabla h^n) -g(\nabla h^{n-1}), \nabla (\delta h^{n+1}) )
& \le \| \nabla \delta h^n\|_2 \| \nabla \delta h^{n+1}\|_2.
\end{aligned}
\end{equation}
Collecting the estimates, we have
\begin{equation}
\begin{aligned}	
& \frac 1 \tau \|\delta h^{n+1}\|^2+ \frac 1{4\tau}\left( \|\delta h^{n+1}\|^2 -  \|\delta h^{n}\|^2+  \|\delta h^{n+1}-\delta h^n\|^2\right) +  \frac 12 \eta^2\|\Delta (\delta h^{n+1} )\|_2^2 \\
&\le  E_n-E_{n+1} + \frac {1}2 \| \nabla (\delta h^{n+1})\|_2^2+
 \| \nabla \delta h^n\|_2 \| \nabla \delta h^{n+1}\|_2.
\end{aligned}
\end{equation}
Thus we obtain
\begin{equation}
\begin{aligned}
& E_{n+1} -E_n + \frac 1{4\tau} \|\delta h^{n+1}\|^2 - \frac 1{4\tau} \|\delta h^{n}\|^2
+(\sqrt{\frac {2\eta^2} \tau} -\frac {1}2) \| \nabla \delta h^{n+1} \|_2^2\\
& \le   \frac {1}2 \| \nabla \delta h^{n+1} \|_2^2 + \frac {1}2 \| \nabla \delta h^n \|_2^2.
\end{aligned}
\end{equation}
Clearly if
\begin{align}
\sqrt{\frac {2\eta^2}{\tau} } \ge \frac {1}2 + 1=\frac 32,
\end{align}
then we have decay of the modified energy. The estimate \eqref{5.13} is obvious.
\end{proof}

\section{Uniform boundedness of energy for any $\tau>0$: the BDF2 case}

\begin{thm}[Uniform boundedness of energy for arbitrary time step]\label{thm6.1}
Consider the scheme \eqref{eq:sch2}. Assume  $h^0$, $h^1 \in H^2(\mathbb T^2)$ 
satisfy
\begin{align} \label{h1h0}
\int_{\mathbb T^2} h^1 dx = \int_{\mathbb T^2} h^0 dx,
\end{align}
and 
\begin{align} \label{h1h0.1}
\frac 1 {\tau} \| h^1- h^0\|_2^2 \le \alpha_2,
\end{align}
where $\alpha_2>0$ is a constant.

Then  for any $\tau>0$, it holds that
\begin{align}
\sup_{n\ge 2} (\| h^n \|_{2} + \| \Delta h^n \|_2 ) \le B_1<\infty,
\end{align}
where $B_1>0$ depends only on ($h^0$, $h^1$, $\eta$, $\alpha_2$).  Note that $B_1$ is independent of $\tau$.

\end{thm}
\begin{rem}
Note that the assumption \eqref{h1h0} is quite reasonable since the mean of $h$ is preserved in time
for the PDE solution. If we compute $h^1$ using the first order scheme \eqref{s3.1}, then it is easy
to check that \eqref{h1h0} holds. The assumption \eqref{h1h0.1} is also quite reasonable
since typically $h^1- h^0 = O(\tau)$. 
\end{rem}
\begin{proof}
By using \eqref{h1h0} and an induction argument, we have
\begin{align}
\int_{\mathbb T^2} h^n dx = \int_{\mathbb T^2} h^0 dx, \qquad\forall\, n\ge 1.
\end{align}
Denote the average of $h^0$ as $\bar h$ and denote
\begin{align}
y^n = h^n - \bar h.
\end{align}
It is not difficult to check that $y^n$ evolves according to the same scheme \eqref{eq:sch2} where $h^n$ is replaced
by $y^n$.  Thus with no loss we can assume all $h^n$ has mean zero.
Note that we may assume $\tau>\alpha_1 \eta^2$ since the case $0<\tau \le \alpha_1 \eta^2$ is
already covered by Theorem \ref{thm5.1}. With some minor change of notation, our desired result
then follows from Theorem \ref{thm6.1a}.
\end{proof}

Assume $f^n = (f^n_1, f^n_2)$, $n\ge 1$ is a given sequence of functions on $\mathbb T^2$.
Let $u^n$ evolve according to the scheme:
\begin{align} \label{uS0}
\frac {3 u^{n+1} -4 u^n +u^{n-1}} {\tau} =-\Delta^2 u^{n+1}+ \nabla \cdot f^n, \qquad\, n\ge 1,
\end{align}
We have the following uniform boundedness result.
\begin{thm} \label{thm6.1a}
Consider the scheme \eqref{uS0} with $\tau\ge \tau_0>0$. Assume $u^0\in H^2(\mathbb T^2)$,
$u^1 \in H^2(\mathbb T^2)$ and have mean zero. Suppose
\begin{align}
\sup_{n\ge 1} \| f^n \|_2 \le  A_0<\infty.
\end{align}
We have
\begin{align}
\sup_{n\ge 2} (\| u^n \|_2 + \| \Delta u^n \|_2) \le  A_1<\infty,
\end{align}
where $A_1>0$ depends only on ($\tau_0$, $\eta$, $A_0$, $u^0$, $u^1$).
\end{thm}
\begin{proof}
We first rewrite \eqref{uS0} as
\begin{align} \label{6.8t0}
u^{n+1}= 4T u^n - T u^{n-1} + \tau T \nabla \cdot f^n,
\end{align}
where $T= (3+\tau \Delta^2)^{-1}$. One should note that since we are working with mean-zero
functions, the operator $T$ admits a natural spectral bound, namely
\begin{align}
|\widehat{T}(k) | \le \frac 1 {3 +\tau} \le \frac 1 {3+\tau_0},\quad
\tau |\widehat{T}(k)| \le 1,
 \qquad\forall\, 0\ne k \in \mathbb Z^2.
\end{align}

We now define a pair of Fourier multipliers $r_+$, $r_-$ by
\begin{align}
\widehat{T_+}(k) = \begin{cases}
2\biggl(\widehat{T}(k) +\sqrt{ (\widehat T(k))^2- \frac 14 \widehat T (k) } \biggr),
\qquad \text{if $\widehat T(k) \ge \frac 14$}; \\
2\biggl(\widehat{T}(k) +i \sqrt{ \frac 14 \widehat T(k) - (\widehat T(k) )^2} \biggr),
\qquad \text{if $\widehat T(k) < \frac 14$};
\end{cases} \\
\widehat{T_-}(k) = \begin{cases}
2\biggl(\widehat{T}(k) -\sqrt{ (\widehat T(k))^2- \frac 14 \widehat T (k) }\biggr),
\qquad \text{if $\widehat T(k) \ge \frac 14$}; \\
2\biggl(\widehat{T}(k) -i \sqrt{ \frac 14 \widehat T(k) - (\widehat T(k) )^2}\biggr),
\qquad \text{if $\widehat T(k) < \frac 14$}.
\end{cases}
\end{align}
It is not difficult to check that
\begin{align}
&\sup_{0\ne k\in \mathbb Z^2} |\widehat{T_+}(k)| \le \theta_0<1;  \notag \\
&\sup_{0\ne k\in \mathbb Z^2} |\widehat{T_-}(k)| \le \theta_0<1,
\end{align}
where $\theta_0$ depends only on $\tau_0$. Furthermore, we have
\begin{align}
u^{n+1}- T_- u^n = T_+ ( u^n -T_- u^{n-1}) +  \tau T\nabla \cdot f^n, \qquad\forall\, n\ge 1.
\end{align}
It follows that
\begin{align}
\| u^{n+1}- T_- u^n\|_2 &\le \theta_0 \| u^n -T_- u^{n-1} \|_2+  A_0 \notag \\
& \le \cdots (\text {iterating in $n$} )\notag \\
& \le \theta_0^n \| u^1 - T_- u^0\|_2+ \frac {A_0}{1-\theta_0} \notag \\
& \le  B_1,
\end{align}
where $B_1$ depends only on ($u^0$, $u^1$, $\theta_0$, $A_0$).

We then obtain
\begin{align}
\| u^{n+1} \|_2 \le \theta_0 \| u^n \|_2 +B_1.
\end{align}
Iterating in $n$ yields the uniform $L^2$ bound:
\begin{align}
\sup_{n\ge 1} \| u^{n+1} \|_2 \le B_2,
\end{align}
where $B_2>0$ depends only on ($u^0$, $u^1$, $\eta$, $\tau_0$, $A_0$).
By using \eqref{6.8t0}, we  obtain the $H^2$-bound.
\end{proof}

\section{Numerical experiments}
In this section, we carry out several numerical simulations of the sinc-type MBE model  with side by side
 comparisons  with the classical MBE model with double well potential. The new sinc-type MBE model
 appears to be much more stable than the classical MBE model whilst producing remarkably similar
 energy landscapes.

More precisely, we will show that the  energy of the new model will always stay uniformly bounded, which is in good agreement with the rigorous results proved in Theorem \ref{thm4.1} and \ref{thm6.1}.
Note that in the classical MBE model with double well potential, this issue is completely
 intractable for large time steps unless one adds additional stabilization terms or introduces additional fictitious dynamics.
Moreover, we verify that under very mild and nearly optimal restrictions on the time step $\tau$, the energy dissipation property can always be ensured. These bounds are  very close to the theoretical predictions as shown
in Theorem \ref{thm3.1} (for first order IMEX) and Theorem \ref{thm5.1} (for second order BDF2).

In the following numerical experiments, we use the pseudo-spectral method with the number of Fourier modes $N_x\times N_y = 256\times 256$ for spatial discretization, .

\begin{example} \label{exam2DMBEsin}
{\em Compare the solutions of the classical MBE equation
\begin{align}\label{eq:mbe_class}
& \partial_t h = -\eta^2 \Delta^2 h  -\nabla \cdot ( (1-|\nabla h|^2)\nabla h), \quad\mbox{on }\mathbb T^2=[-\pi,\pi]^2,
\end{align}
and the 2D sinc-type MBE equation
\begin{align}\label{eq:mbe_sinc}
& \partial_t h = -\eta^2 \Delta^2 h   -\nabla \cdot ( \mathrm{sinc}( |\nabla h|)\nabla h), \quad\mbox{on }\mathbb T^2=[-\pi,\pi]^2,
\end{align}
with initial condition $h_0(x,y)= 0.1(\sin(3x)\sin(2y)+\sin(5x)\sin(5y))$.
}
\end{example}

Firstly we compare the numerical phases of the classical MBE equation and the sinc-type MBE equation computed by the first-order IMEX scheme respectively in Figure \ref{fig:imex1} and \ref{fig:imex2}.
Here we take $\eta = 0.1$ and the time step $\tau = 0.001$.
The corresponding energy evolutions are illustrated in Figure \ref{fig:comp_energy}.

\begin{figure}[!h]
\centering
\includegraphics[width=0.33\textwidth]{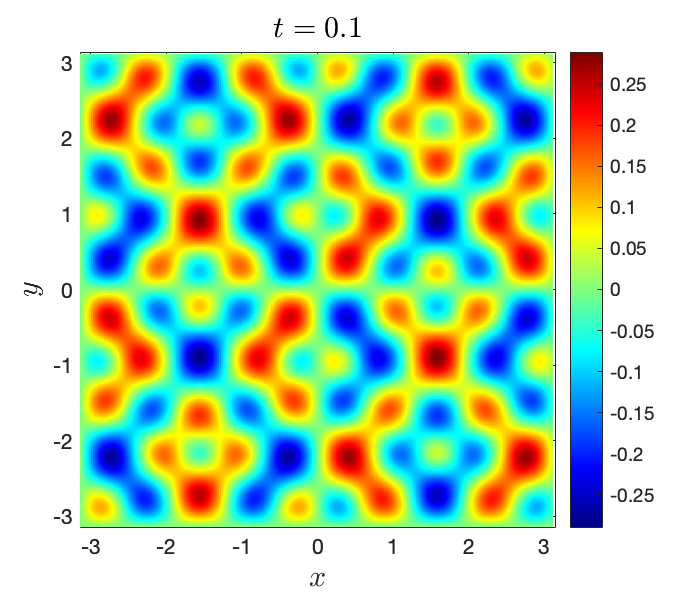}
\includegraphics[width=0.33\textwidth]{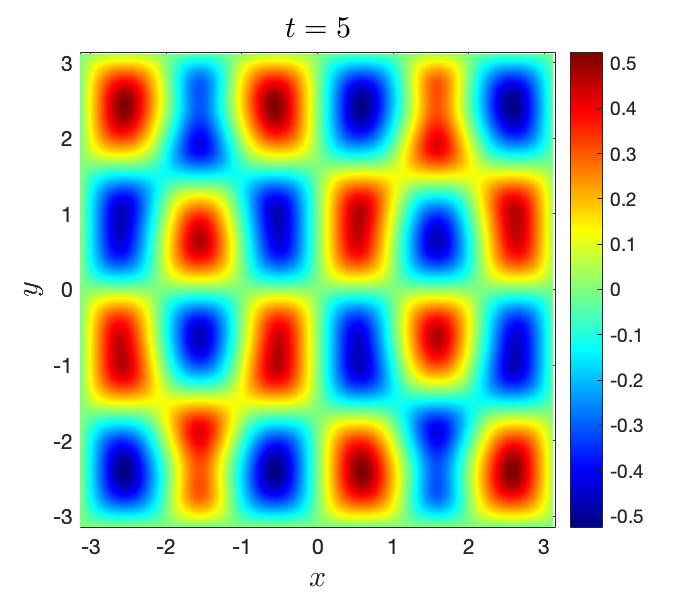}\\
\includegraphics[width=0.33\textwidth]{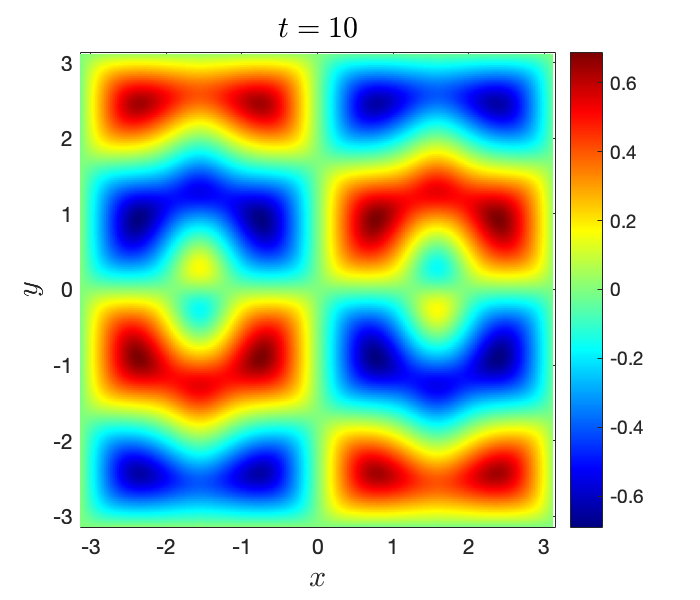}
\includegraphics[width=0.33\textwidth]{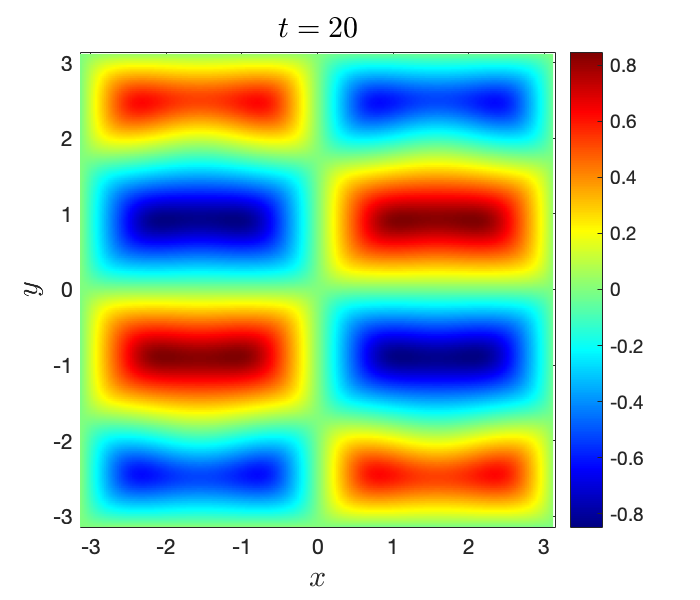}
\caption{\small Example \ref{exam2DMBEsin}: Dynamics of  2D classical MBE equation \eqref{eq:mbe_class} using the first-order IMEX scheme where $\eta = 0.1,~\tau= 0.001,~N_x=N_y = 256$. }\label{fig:imex1}
\centering
\includegraphics[width=0.33\textwidth]{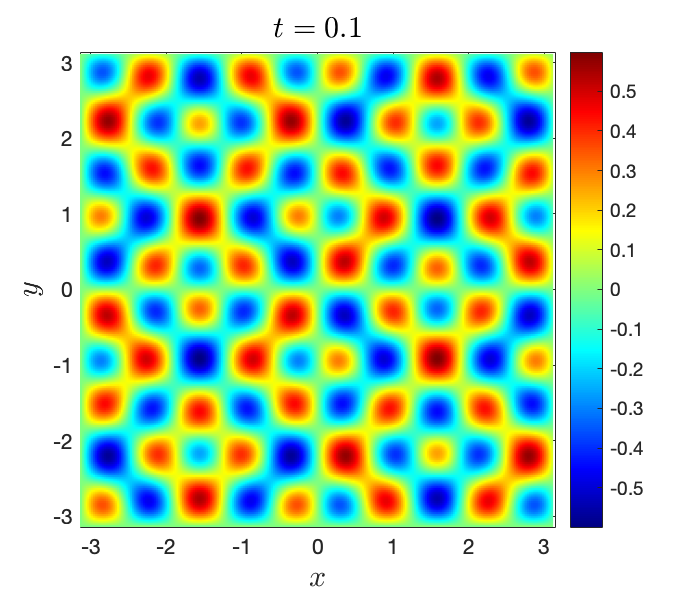}
\includegraphics[width=0.33\textwidth]{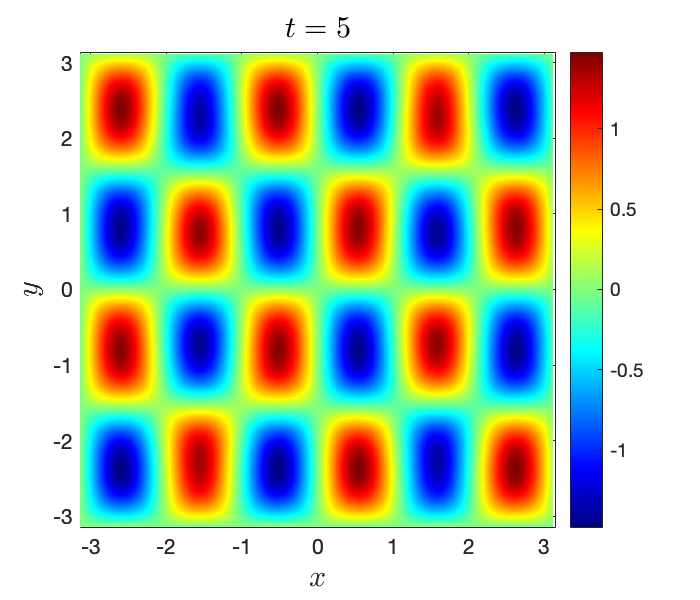}\\
\includegraphics[width=0.33\textwidth]{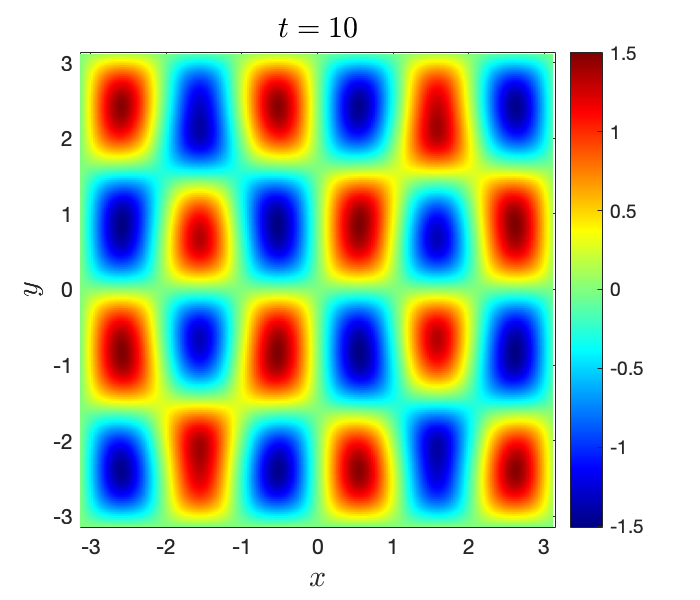}
\includegraphics[width=0.33\textwidth]{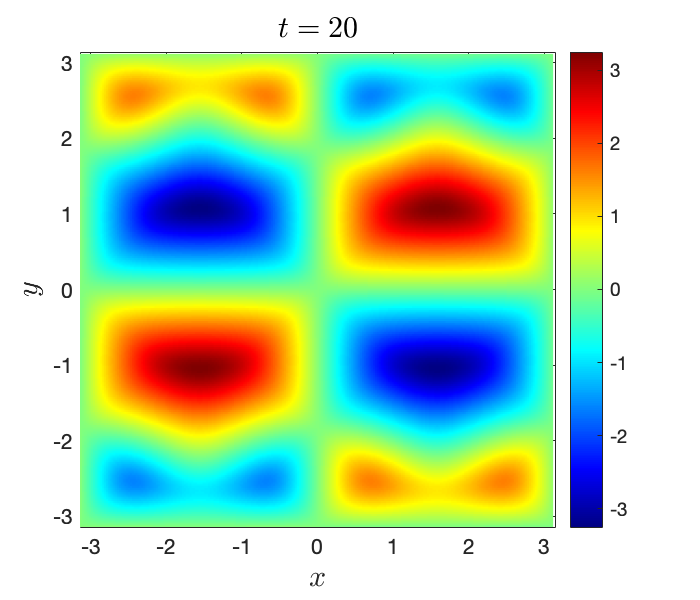}
\caption{\small Example \ref{exam2DMBEsin}: Dynamics of 2D sinc-type MBE equation \eqref{eq:mbe_sinc} using the first-order IMEX scheme where $\eta = 0.1,~\tau= 0.001,~N_x=N_y = 256$.  }\label{fig:imex2}
\end{figure}

\begin{figure}[!h]
\centering
\includegraphics[width=0.43\textwidth]{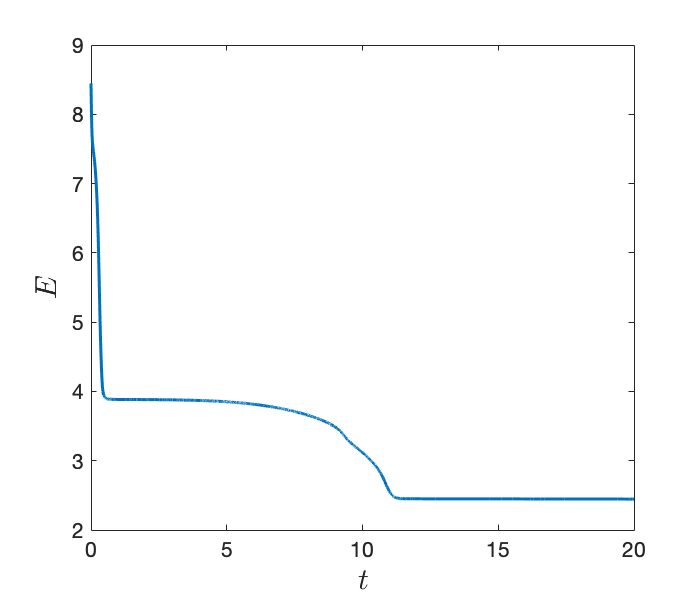}
\includegraphics[width=0.43\textwidth]{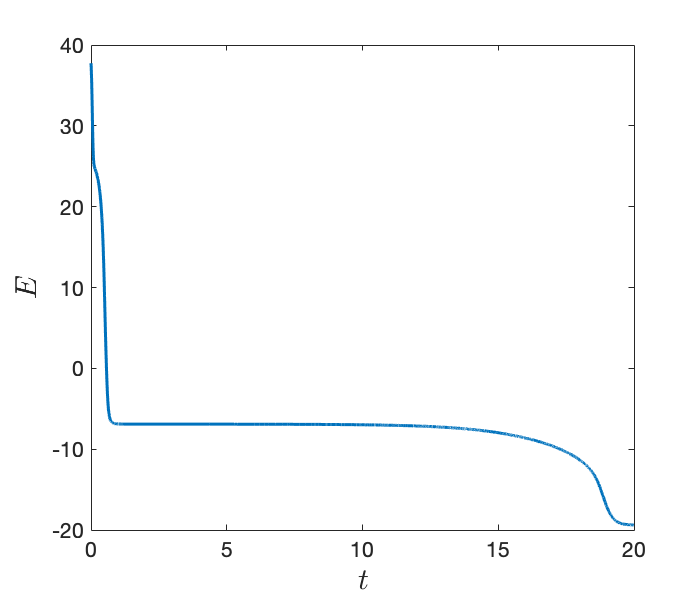}\\
\caption{\small Example \ref{exam2DMBEsin}: Energy evolutions of 2D classical MBE equation \eqref{eq:mbe_class} (left) and the sinc-type MBE equation \eqref{eq:mbe_sinc} (right) computed by the first-order IMEX scheme where $\tau= 0.001,~N_x=N_y = 256$.  }\label{fig:comp_energy}
\end{figure}

Secondly, we compare the energy evolutions of the classical MBE and the sinc-type MBE for different time step $\tau$ (see Figure \ref{fig:blowup}).
It is observed that for the classical MBE model, the energy of IMEX scheme will blow up when $\tau\geq 0.1$.
However, for the sinc-type MBE model, the energy will not blow up (despite of the oscillation) even when $\tau =10$.
This is consistent with our analysis on uniform boundness (cf. Theorem \ref{thm4.1}).

\begin{figure}[!h]
\centering
\includegraphics[width=0.43\textwidth]{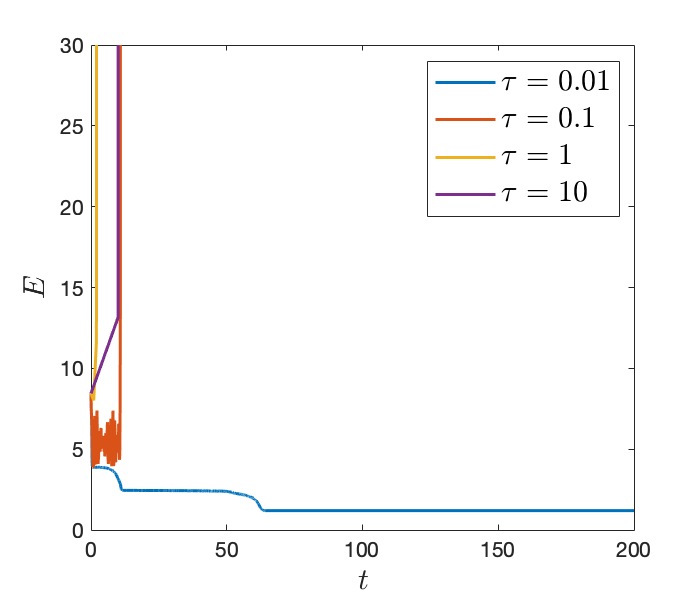}
\includegraphics[width=0.43\textwidth]{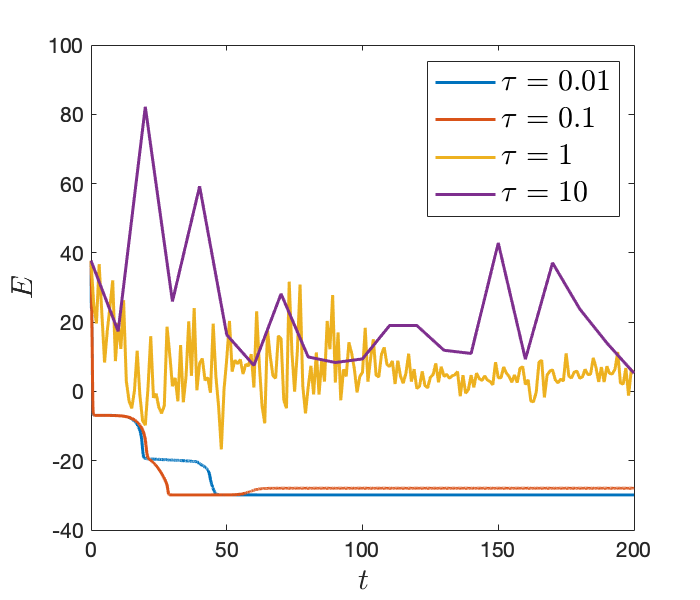}\\
\caption{\small Example \ref{exam2DMBEsin}: Energy blowup phenomenon of 2D classical MBE equation \eqref{eq:mbe_class} (left) and no energy blowup of the sinc-type MBE equation \eqref{eq:mbe_sinc} (right) for different $\tau$, computed by the first-order IMEX scheme with $N_x=N_y = 256$.  }\label{fig:blowup}
\end{figure}

Thirdly, in Table \ref{tab1}, we compare the threshold $\tau_{\mathrm c}$ of time step  preserving the energy dissipation, for the classical and sinc-type MBE models with different quantities of $\eta$ and different numerical schemes.
To be precise, we say that the energy dissipation is not preserved if there exists $0<t_n\leq T$ such that
\begin{equation}
E^{n}-E^{n-1}\geq {\texttt{Tol}} =  10^{-12},
\end{equation}
where ${\texttt{Tol}}$ is taken to be a mediocrely small positive number because we take account of machine error.
As a consequence, $\tau_{\mathrm c}$ is defined by
\begin{equation}
\tau_{\mathrm c} = \sup\left \{\tau>0\mid E^n<E^{n-1}+{\texttt{Tol}},~\forall~ 0<t_n\leq T\right\}.
\end{equation}
It can be observed in Table \ref{tab1} that the sinc-type MBE model seems  to always have larger threshold $\tau_{\mathrm c}$ than the classical MBE model.
According to Theorem \ref{thm3.1}, for the IMEX scheme of sinc-type MBE model, the following inequality shall hold
\begin{equation}
\tau_{\mathrm c} \ge 8 \eta^2,
\end{equation}
which is verified by Table \ref{tab1}.
Similarly, according to Theorem \ref{thm5.1}, for the BDF2 scheme of sinc-type MBE model, the following inequality shall hold
\begin{equation}
\tau_{\mathrm c} \ge \frac 89 \eta^2,
\end{equation}
which is also verified by Table \ref{tab1}.

\begin{table}
\centering
\caption{Threshold $\tau_{\mathrm c}$ of time step  preserving energy dissipation, for the classical MBE and sinc-type MBE where $T = 200$.}\label{tab1}
\renewcommand\arraystretch{1.25}
\begin{tabular}{lll|ll}
\hline
\cline{1-5}
& \multicolumn{2}{c|}{IMEX} & \multicolumn{2}{c}{BDF2}  \\
\cline{2-3}\cline{4-5}
$\eta^2$ & classical & sinc-type & classical & sinc-type\\
\hline
$0.1$    & $0.6<\tau_{\mathrm c}<0.7$   & $1<\tau_{\mathrm c}<2$  & $0.01<\tau_{\mathrm c}<0.02$  &   $0.06<\tau_{\mathrm c}<0.07$    \\
$0.01$   &  $0.01<\tau_{\mathrm c}<0.02$   & $0.09<\tau_{\mathrm c}<0.1$   & $0.003<\tau_{\mathrm c}<0.004$   &  $0.02<\tau_{\mathrm c}<0.03$      \\
$0.001$   & $0.001<\tau_{\mathrm c}<0.002$ & $0.008<\tau_{\mathrm c}<0.009$   &  $0.0003<\tau_{\mathrm c}<0.0004$ &    $0.001<\tau_{\mathrm c}<0.002$   \\
\hline
\cline{1-5}
\end{tabular}
\end{table}

\begin{example} \label{exam2DMBEsin_square}
{\em Consider the following 2D sinc-type MBE equation for the growth on square symmetry surfaces
\begin{align}\label{eq:mbe_sinc_square}
& \partial_t h = -\eta^2 \Delta^2 h   - (\sin(h_x) )_x - ( \sin(h_y))_y, \quad\mbox{on }\mathbb T^2=[-\pi,\pi]^2,
\end{align}
where $\eta = 0.1$ and the initial condition is a random state by assigning a random number varying from $-0.01$ to $0.01$ to each grid point.
}
\end{example}

The model \eqref{eq:mbe_sinc_square} is derived from replacing $(1-|h_x|^2)h_x$ and $(1-|h_y|^2)h_y$ by $\sin(h_x)$ and $\sin(h_y)$ respectively in the classical MBE equation for the growth on square symmetry surfaces (cf. \cite{xt2006}).
As a consequence, the energy functional of \eqref{eq:mbe_sinc_square} is defined by
\begin{equation}
\mathcal E = \int_{\Omega} \Bigl( \frac 12 \eta^2 |\Delta h|^2 +  \cos(h_x) + \cos(h_y) \Bigr) dx,
\end{equation}
We solve the sinc-type MBE equation \eqref{eq:mbe_sinc_square} using the following BDF2 scheme: $\forall n\geq 1$,
\begin{equation}\label{eq:sch2_square}
\frac{3h^{n+1}-4h^n+h^{n-1}}{2\tau}=-\eta^2\Delta^2 h^{n+1}-2 \nabla \cdot (\sin(h_x^n),\sin(h_y^n))+
\nabla \cdot  ( \sin(h_x^{n-1}),\sin(h_y^{n-1})),
\end{equation}
where the time step $\tau = 0.01$.
In Figure \ref{fig:square2}, the isolines of the free energy
\begin{equation}\label{eq:freeF}
 F =\frac 12 \eta^2 |\Delta h|^2 +  \cos(h_x) + \cos(h_y)
\end{equation}
are plotted at different time.
The corresponding energy evolution is plotted in Figure \ref{fig:square2_energy}.
It can be observed that the energy dissipation property is preserved.

\begin{figure}[!h]
\centering
\includegraphics[width=0.46\textwidth]{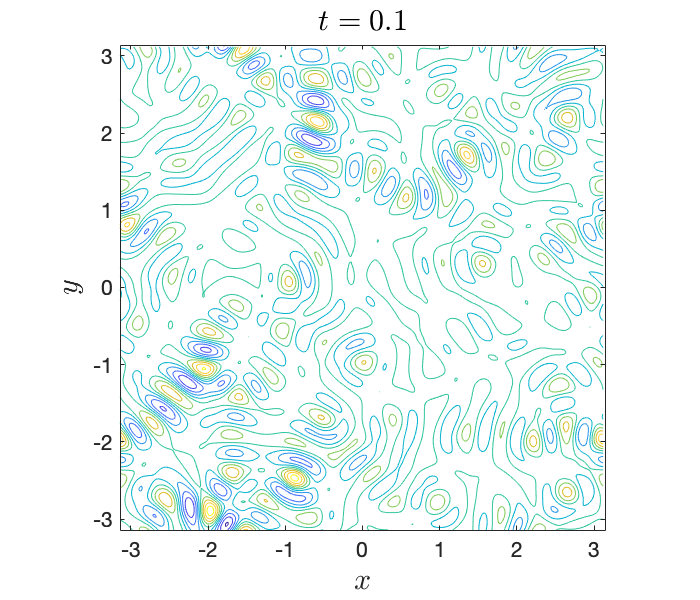}
\includegraphics[width=0.46\textwidth]{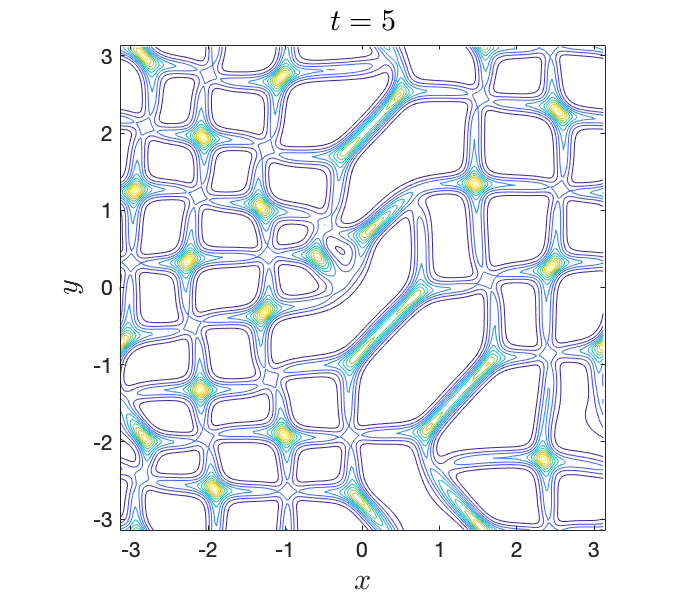}\\
\includegraphics[width=0.46\textwidth]{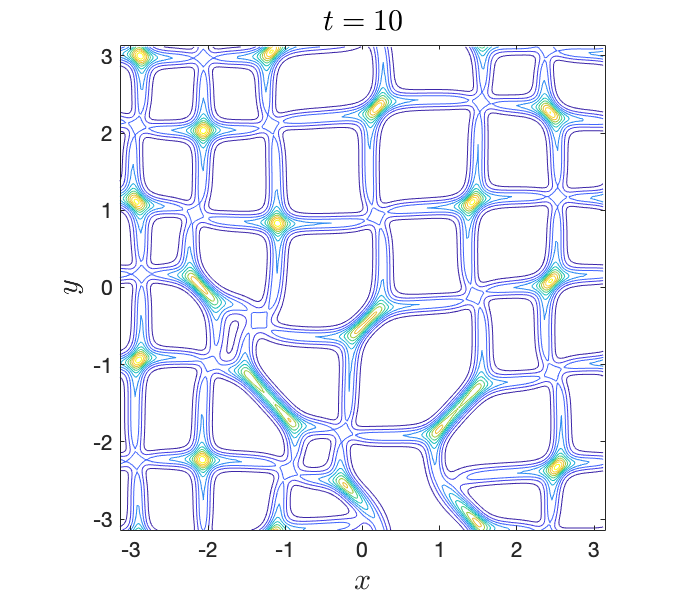}
\includegraphics[width=0.46\textwidth]{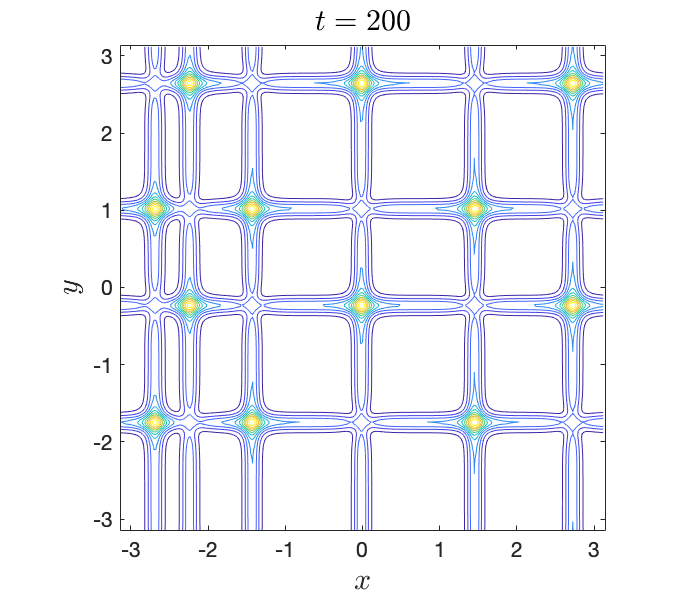}\\
\caption{\small Example \ref{exam2DMBEsin_square}: Isolines of the free energy $ F$ in \eqref{eq:freeF} at different time for the sinc-type MBE model \eqref{eq:mbe_sinc_square}, computed by the BDF2 scheme \eqref{eq:sch2_square} with $\tau = 0.01$, $N_x=N_y = 256$.  }\label{fig:square2}
\end{figure}

\begin{figure}[!h]
\centering
\includegraphics[width=0.43\textwidth]{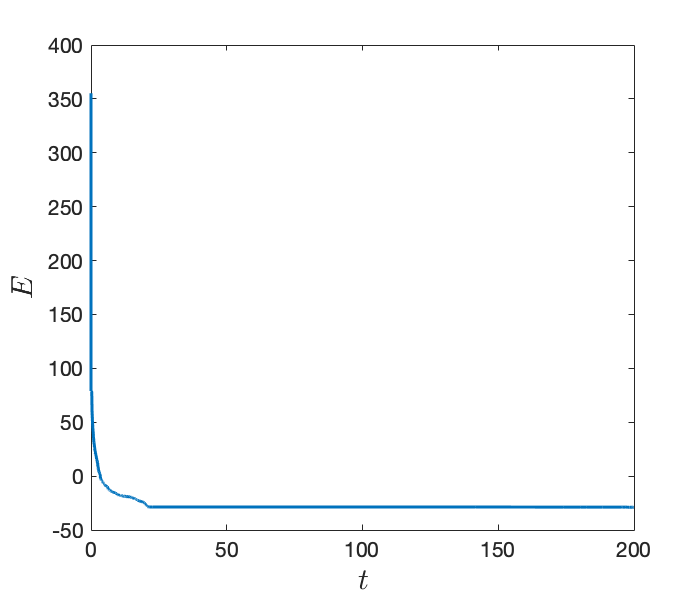}
\caption{\small Example \ref{exam2DMBEsin_square}: Energy evolution of the BDF2 scheme \eqref{eq:sch2_square} with $\tau = 0.01$, $N_x=N_y = 256$.  }\label{fig:square2_energy}
\end{figure}

\section{Concluding remarks}
In this work we introduced a new molecular beam epitaxy model driven by a cosine-type free
energy.  The energy landscape of the new free energy is similar to the classical double-well potential case but has the additional advantage that all its  derivatives are uniformly bounded.   We analyzed two types of numerical
discretization: one is the 1st order IMEX method, and the other is BDF2 in time with explicit treatment of
the nonlinear term. We characterized nearly optimal time step constraints for both the first order and the
second order schemes, and showed energy dissipation in both cases. We introduced two methods to show
that the the energy of the numerical solutions in both schemes are unconditionally uniformly bounded, i.e.
 the obtained upper bound is independent of the time step. To our best knowledge, this kind of results are the first in the literature. We also carried out several numerical
 experiments which show that this new model is far more superior than the existing double well type MBE models.
It is expected that our new theoretical framework
 can be generalized to many other phase-field models with good Lipschitz type nonlinearities. In addition it will
 be interesting to investigate the performance of other numerical schemes on the new model as well as the
 structural stability/instability of the metastable states and steady states.

{\bf Acknowledgement.}
The research of D. Li is supported in part by Hong Kong RGC grant GRF 16307317 and 16309518.
The research of W. Yang is supported by NSFC Grants 11801550 and 11871470.
The research of C. Quan is supported by NSFC Grant 11901281, the Guangdong Basic and Applied Basic Research Foundation (2020A1515010336), and the Stable Support Plan Program of Shenzhen Natural Science Fund (Program Contract No. 20200925160747003).

\frenchspacing
\bibliographystyle{plain}


\begin{thebibliography}{}

\end{thebibliography}


\begin{thebibliography}{99}
\bibitem{a2002}
J.R. Arthur. Molecular beam epitaxy. {\em Surface science}. 500(1-3), (2002), 189-217.	
	
\bibitem{bs1995}
A.L. Barab\'asi, H.E. Stanley. Fractal Concepts in Surface Growth, Cambridge University Press, Cambridge, UK, 1995.
	
\bibitem{bl2015}
J. Bourgain, D. Li. Strong ill-posedness of the incompressible Euler equation in borderline Sobolev spaces, {\em Invent. Math.} 201 (2015), pp. 97-157.
%
%



\bibitem{cv1987}
S. Clarke, D.D. Vvedensky, Origin of reflection high-energy electron-diffraction intensity oscillations during molecular-beam epitaxy: A computational modeling approach. {\em Phys. Rev. Lett.} 58 (1987), pp. 2235-2238.


\bibitem{ca1975}
A.Y. Cho, J.R. Arthur. Molecular beam epitaxy. {\em Progress in solid state chemistry}. 10, (1975), 157-191.


\bibitem{dg1992}
S. Das Sarma, S.V. Ghaisas. Solid-on-solid rules and models for nonequilibrium growth in $2+1$ dimensions.  {\em Physical Review Letters}. 69 (26), 1992, 2651-2654.

\bibitem{d1996}
S. Das Sarma, C.J. Lanczycki, R. Kotlyar, S.V. Ghaisas. Scale invariance and dynamical correlations in growth models of molecular beam epitaxy. {\em Physical Review E}, 53(1), (1996), 359-388.


\bibitem{EH1966}
G. Ehrlich, F.G. Hudda, Atomic view of surface diffusion: Tungsten on tungsten. {\em J. Chem. Phys.} 44 (1966), pp. 1039-1049.


\bibitem{g1997}
L. Golubovi\'c. Interfacial coarsening in epitaxial growth models without slope selection. {\em Phys. Rev. Lett.} 78 (1), (1997), 90-93.


\bibitem{h1951}
C. Herring. Surface tension as a motivation for sintering In: Kingston, W.E. (Ed.) The Physics of powder Metallurgy, McGraw-Hill, New York.


\bibitem{johnson1994}
M.D. Johnson, C. Orme, A.W. Hunt, D. Graff, J. Sudijono, L.M. Sander, B.G. Orr. Stable and unstable growth in molecular beam epitaxy. {\em Physical review letters}. 72(1), (1994), 116-119.

\bibitem{kd1994}
J.M. Kim, S. Das Sarma. Discrete models for conserved growth equations. {\em Physical review letters}. 72(18), (1994), 2903-2906.



\bibitem{km1994}
R. Kohn, S. M\"uller. Surface energy and microstructure in coherent phase transitions. {\em Comm. Pure Appl. Math.} 47, (1994), 405-435.

\bibitem{ky2003}
R. Kohn, X. Yan. Upper bounds on the coarsening rate for an epitaxial growth model. Comm. Pure Appl. Math. 56(11), (2003), 1549-1564.

\bibitem{k1997}
J. Krug, Origins of scale invariance in growth processes. {\em Adv. in Phys.} 46 (1997), pp. 139-282.

\bibitem{ll2003}
B. Li, J.G. Liu. Thin film epitaxy with or without slope selection. {\em European Journal of Applied Mathematics}, 14(6) (2003), 713-743.

\bibitem{ld1991}
Z.W. Lai, S. Das Sarma. Kinetic growth with surface relaxation: Continuum versus atomistic models. {\em  Physical review letters}, 66(18), (1991), 2348-2351.


\bibitem{Li2013}
D. Li. On a frequency localized Bernstein inequality and some generalized Poincar\'e-type inequalities. {\em Mathematical Research Letters}. 20.5 (2013): 933-945.

\bibitem{lqt2016}
D. Li, Z.H. Qiao, T. Tang. Characterizing the stabilization size for semi-implicit Fourier-spectral method to phase field equations. {\em SIAM J. Numer. Anal.} 54 (2016), no. 3, 1653-1681.

\bibitem{lqt2017}
D. Li, Z.H. Qiao, T. Tang. Gradient bounds for a thin film epitaxy equation. {\em Journal of Differential Equations}, 262 (2017), no. 3, 1720-1746.




\bibitem{lwy2020}
D. Li, F. Wang, K. Yang. An improved gradient bound for 2D MBE. {\em Journal of Differential Equations}, 269.12 (2020): 11165-11171.

\bibitem{lt2021}
D. Li, T. Tang. Stability of the Semi-Implicit Method for the Cahn-Hilliard Equation with Logarithmic Potentials. {\em Ann. Appl. Math.}, 37 (2021), 31-60.

\bibitem{lt2021-a}
D. Li, T. Tang. Stability analysis for the Implicit-Explicit discretization of the Cahn-Hilliard equation. arXiv:2008.03701.

\bibitem{m1957}
W.W. Mullins. Theory of thermal grooving. {\em Journal of Applied Physics.} 28(3), (1957), 333-339.

\bibitem{onsager1931}
L. Onsager. Reciprocal relations in irreversible processes. I. {\em Physical review}. 37(4), (1931), 405-426.

\bibitem{onsager1931-2}
L. Onsager. Reciprocal relations in irreversible processes. II. {\em Physical review}. 38(12), (1931), 2265-2279.


\bibitem{ors1999}
M. Ortiz, E. Repetto, H. Si. A continuum model of kinetic roughening and coarsening in thin films. {\em J. Mech. Phys. Solids}. 47, (1999), 697-730.


\bibitem{pb1996}
R.K. Pathria, P.D. Beale. Statistical mechanics, 1996. Butter worth, 32.

\bibitem{pgmo200}
P. Politi, G. Grenet, A. Marty, A. Ponchet, J. Villain. Instabilities in crystal growth by atomic or molecular beams. {\em Phys. Reports}, 324, (2000), 271-404.


\bibitem{pv}
P. Politi, J. Villain. Ehrlich-Schwoebel instability in molecular-beam epitaxy: A minimal model. {\em Phys. Rev. B}, 54(7), (1996), 5114-5129.


\bibitem{SSR}
M. Schneider, I.K. Schuller, A. Rahman, Epitaxial growth of silicon: A moleculardynamics
simulation. {\em Phys. Rev. B}, 36 (1987), pp. 1340-1343.


\bibitem{SY10}
J. Shen and X. Yang. {\emph Numerical approximations
of Allen-Cahn and Cahn-Hilliard equations.}
Discrete Contin. Dyn. Syst. A, 28 (2010),
1669--1691.



\bibitem{swww2012}
J. Shen, C. Wang, X. Wang, S.M. Wise. Second-order convex splitting schemes for gradient flows with Ehrlich-Schwoebel type energy: Application to thin film epitaxy. {\em SIAM J. Numer. Anal.}, 50 (2012), pp. 105-125.

\bibitem{sp1994}
M. Siegert, M. Plischke. Slope selection and coarsening in molecular beam epitaxy. {\em Phys. Rev. Lett.} 73(11), (1994), 1517-1520.

\bibitem{SongShu17}
H. Song and C.W Shu.
Unconditional Energy Stability Analysis of a Second Order Implicit-Explicit Local Discontinuous Galerkin Method for the Cahn-Hilliard Equation.
{\em Journal of Scientific Computing.}
 volume 73 (2017), pages 1178--1203.


\bibitem{v1991}
J. Villain, Continuum models of critical growth from atomic beams with and without desorption,
{\em J. Phys. I.}, 1 (1991), pp. 19-42.


\bibitem{xt2006}
C.J. Xu, T. Tang. Stability analysis of large time-stepping methods for epitaxial growth models. {\em SIAM J. Numer. Anal.} 44 (2006), no. 4, 1759-1779.

\bibitem{z1995}
A. Zangwill. Theory of growth-induced surface roughness. In: Hatwater, H.A., Thompson, C. (Eds.), Microstructural Evolution of Thin Films. Academic Press, New York.

\bibitem{z1996}
A. Zangwill. Some causes and a consequence of epitaxial roughening. {\em Journal of Crystal Growth}. 163, (1996), 8-21.

\end{thebibliography}

\end{document}